\RequirePackage{fix-cm}
\documentclass[smallextended]{svjour3}       
\smartqed  
\usepackage{graphicx}
\usepackage {color}
\usepackage {tikz}

\smartqed  
\usepackage{graphicx}
\usepackage{amsmath}
\usepackage{amssymb}
\usepackage{mathtools}
\usepackage{hyperref}

\usepackage{booktabs}
\usepackage{pifont} 
\usepackage{adjustbox}
\usepackage{multirow} 

\usepackage{subcaption}
\usepackage{algorithm}
\usepackage{algorithmic}
\usepackage{booktabs} 
\usepackage{amsmath}
\allowdisplaybreaks

\newtheorem{assumption}{Assumption}

\begin{document}

\title{Lipschitz-Free Mirror Descent Methods for Non-Smooth Optimization Problems}


\titlerunning{Lipschitz-Free Mirror Descent Methods for Non-Smooth Optimization Problems}    

\author{Bowen Yuan, Mohammad S. Alkousa}
\authorrunning{Bowen Yuan, Mohammad S. Alkousa} 

\institute{Bowen Yuan \at School of Mathematical Sciences, Beihang University, Beijing 100191, People's Republic of China.
\email{yuanbowen@buaa.edu.cn}
\\
Mohammad S. Alkousa \at Innopolis University, Innopolis, Universitetskaya Str., 1, 420500, Russia.
\email{m.alkousa@innopolis.ru}
	  }     

\date{Received: date / Accepted: date}
\maketitle

\begin{abstract}
The part of the analysis of the convergence rate of the mirror descent method that is connected with the adaptive time-varying step size rules due to Alkousa et al. \cite{alkousa2024mirror} is corrected. Moreover, a Lipschitz-free mirror descent method that achieves weak ergodic convergence is presented, generalizing the convergence results of the mirror descent method in the absence of the Lipschitz assumption. 
\keywords{Mirror descent \and Mirror-C descent \and Time-varying step size \and Non-smooth convex optimization \and Optimal convergence rate}
	\subclass{90C25, 90C30}
\end{abstract}

\section{Introduction}
{\color{black}Mirror descent method originated in \cite{Nemirovsky1983Complexity,Nemirovskii1979efficient} and was later analyzed in \cite{Beck2003Mirror}. It is considered as the non-Euclidean extension of the standard subgradient method by employing a nonlinear distance function with an optimal step size in the nonlinear projection step. The mirror descent method is also applicable to optimization problems in Banach spaces where gradient descent is not \cite{article:doan_2019}. An extension of the mirror descent method for constrained problems was proposed in \cite{Nemirovsky1983Complexity,article:beck_comirror_2010}. This method is used in many applications, see \cite{applications_tomography_2001,article:Nazin_2014} and references therein.}

One of the focuses of the research on the projected subgradient method and the mirror descent method is the convergence rate of the algorithms under different step size rules. When the step size is fixed as $\gamma_k = O(1)/\sqrt{N}$, where $N$ denotes the maximum number of iterations, the optimal convergence rate of the projected subgradient method has been established at $O\left(N^{-1/2}\right)$. Moreover, the projected subgradient method has been considered to have only sub-optimal convergence rate $O(N^{-1/2}\log N)$ for time-varying step sizes $\gamma_k=O(1)/\sqrt{k}$, which is well-documented in \cite{nesterov2004lectures,nesterov2018lectures,bubeck2015convex}.

By proposing a new weighting scheme for the iteration points, which is called weak ergodic convergence, Z. Zhu et al. \cite{Zhu} first proved that the projected subgradient method can also achieve the optimal convergence rate of $O(N^{-1/2})$ under the time-varying step size rules. Based on the result of \cite{Zhu}, Alkousa et al. proved that the optimal convergence rate under the time-varying step sizes still holds for the mirror descent method \cite{alkousa2024mirror}.

By devising a novel class of time-varying step size rules, Y. Xia et al. \cite{Lipschitz-free} successfully extended the convergence results of the projected subgradient method to the non-Lipschitz case. In this paper, we aim to further generalize the findings of \cite{Lipschitz-free} to the mirror descent method. Specifically, we will establish the optimal convergence results of the mirror descent method under more general conditions, without relying on the listed Assumption \ref{ass} below.

\section{{\color{black}Fundamentals for} the mirror descent method}

{\color{black}
Let $(\mathbf{E},\|\cdot\|)$ be a normed finite-dimensional vector space, with an arbitrary norm $\|\cdot\|$, and $\mathbf{E}^*$ be the conjugate space of $\mathbf{E}$ with the following norm
$$
    \|y\|_{*}=\max\limits_{x \in \mathbf{E}}\{\langle y,x\rangle: \|x\|\leq1\},
$$
where $\langle y,x\rangle$ is the value of the continuous linear functional $y \in \mathbf{E}^*$ at $x \in \mathbf{E}$.

Let $Q \subset \mathbf{E}$ be a compact convex set, and 
$\psi: Q \longrightarrow \mathbb{R}$ be a proper closed differentiable and $\sigma$-strongly convex (called prox-function or distance generating function) with $\sigma > 0 $. The corresponding Bregman divergence is defined as 
$$
    V_{\psi} (x, y) = \psi (x) - \psi (y) - \langle \nabla \psi (y),  x -y \rangle \quad \forall x , y \in Q. 
$$
For the Bregman divergence, it holds the following inequality
\begin{equation}\label{eq_breg}
    V_{\psi} (x, y) \geq \frac{\sigma}{2} \|y - x\|^2 \quad \forall x, y \in Q. 
\end{equation}

In what follows, we denote the subdifferential of $f$ at $x$ by $\partial f(x)$, and the subgradient of $f$ at any point $x$ by $\nabla f(x) \in \partial f(x)$.   Let $\operatorname{dom} f$ denote the domain of the function $f$, and $\operatorname{dom} (\partial f)$ denote the set of points of subdifferentiability of $f$, i.e., 
$$
    \operatorname{dom} (\partial f) = \left\{ x \in  \mathbf{E} : \partial f(x) \ne \emptyset  \right\}. 
$$

The following identity, known as the three-point identity, is essential in the analysis of the mirror descent method.

\begin{lemma}{\bf{(Three-point lemma)}\cite{beck2017first}}\label{three_point}
Suppose that $\psi: \mathbf{E} \longrightarrow (- \infty, \infty]$ is proper
closed, convex, and differentiable over $\operatorname{dom}(\partial \psi)$. Let $a, b \in \operatorname{dom}(\partial(\psi))$ and $c \in \operatorname{dom} (\psi)$. Then it holds 
\begin{equation}\label{eq_three_points}
    \langle\nabla \psi(b)-\nabla \psi(a), c-a\rangle=V_{\psi}(c,a)+V_{\psi}(a, b)-V_{\psi}(c, b).
\end{equation}
\end{lemma}

The following lemma is an extension of the second prox theorem to the case of non-Euclidean distances.
\begin{lemma}{\bf{ (Non-Euclidean second prox theorem)}\cite{beck2017first}}\label{prox_non}
    Let 
\begin{enumerate}
\item $\psi: \mathbf{E} \longrightarrow (-\infty, \infty]$ be a proper closed and convex function differentiable over $\operatorname{dom}(\partial \psi)$,

\item $\varphi: \mathbf{E} \longrightarrow (-\infty, \infty]$ be a proper closed and convex function satisfying  $\operatorname{dom}(\partial \varphi) \subseteq \operatorname{dom}(\partial \psi)$,

\item $\psi + \mathbb{I}_{\operatorname{dom}(\varphi)}$ be a $\sigma$-strongly convex, where $\mathbb{I}_{A}$ is the indicator function of the set $A$. 
\end{enumerate}
Assume that $b \in \operatorname{dom}(\partial \psi)$, and let $a$ be defined by
$$
    a = \arg\min_{x \in \mathbf{E}} \left\{ \varphi(x) + V_{\psi}(x, b) \right\}. 
$$
Then 
we have
\begin{equation}\label{eq_main_lemma}
    \left\langle \nabla \psi(b) - \nabla \psi(a), u - a \right\rangle \leq \varphi(u) - \varphi(a)	\quad \forall u \in \operatorname{dom}(\varphi).
\end{equation}
\end{lemma}

\bigskip 

\noindent
\textbf{Fenchel-Young inequality.} For any $a \in \mathbf{E}, b \in \mathbf{E}^*$, it holds the following inequality
\begin{equation}\label{Fenchel_Young_ineq}
    \left|  \langle a,b \rangle\right| \leq \frac{\|a\|^2}{2 \lambda} + \frac{\lambda\|b\|_*^2}{2} \quad \forall \lambda > 0 .
\end{equation}
}



{\color{black}
\section{Lipschitz-free mirror descent method for constrained problem}
}

In this section, we consider the convex optimization problem
\begin{equation}\label{convex_problem}
    \min_{x\in Q} f(x)
\end{equation}
where $f : Q \longrightarrow \mathbb{R}$ is a convex continuous function.

The mirror descent method can be represented by the steps in Algorithm \ref{alg_mirror_descent} below, which is excerpted from \cite{alkousa2024mirror}.

\begin{algorithm}[!ht]
\caption{Mirror Descent Method.}\label{alg_mirror_descent}
\begin{algorithmic}[1]
\REQUIRE Step sizes $\{\gamma_k\}_{k \geq 1}$,  initial point $x^1  \in Q$, number of iterations $N$.
\FOR{$k= 1, 2, \ldots, N$}
\STATE Calculate $\nabla f(x^k) \in \partial f(x^k)$,
\STATE $x^{k+1} = \arg\min_{x \in Q} \left\{ \langle x, \nabla f(x^k)  \rangle + \frac{1}{\gamma_k} V_{\psi}(x,x^k) \right\} $.
\ENDFOR
\end{algorithmic}
\end{algorithm}

\subsection{Incorrectness in the analysis for adaptive step size in \cite{alkousa2024mirror}}

{\color{black}In the main convergence result in \cite{alkousa2024mirror}, it is necessary to require that $\{\gamma_k\}_{k \geq 1}$ is a positive non-increasing sequence of step sizes. In \cite[Corollary 3.3]{alkousa2024mirror} (also in \cite[Corollary 4.3]{alkousa2024mirror} for the composite problems), Alkousa et al. analyzed the optimal convergence rate of the mirror descent method with the adaptive time-varying step size rule (which we will call the Nesterov step size)
\begin{equation}\label{nesterov_step_size}
    \gamma_k = \frac{\sqrt{2 \sigma}}{\|\nabla f(x^k)\|_* \sqrt{k}}, \quad  k = 1, 2, \ldots, N.
\end{equation} }



But the Nesterov step sizes \eqref{nesterov_step_size} are not necessarily non-increasing, see Example \ref{counter_nesterov} as a counterexample in one-dimensional Euclidean space. Fortunately, this conflict can be avoided by making a modification to the step size setting. One of the goals of this note is to provide a corrective proof of the optimal convergence rate of the mirror descent method under the adaptive time-varying step size rules.

\begin{example}\label{counter_nesterov}
Consider the convex optimization problem \eqref{convex_problem}in one-dimensional Euclidean space. Let the function \(f(x)=\frac{1}{2}x^{2}\) and the distance generating function $\psi(x)=\frac{1}{2}x^{2}$. In this case, the mirror descent method is equivalent to the projected subgradient method.
Set the initial point as \(x^1 = 10\), the constraint set \(Q\) as \(\vert x\vert\leq10\). 
Select the Nesterov step sizes defined in \eqref{nesterov_step_size}.
The experimental results in Table \ref{tab:vertical_table_nesterov} demonstrate that the Nesterov step sizes \(\gamma_k\) are not necessarily non-increasing. 

\begin{table}[h]
    \centering
    \renewcommand{\arraystretch}{1.3}  
    \vspace{-0.5em}  
    \caption{Nesterov step sizes in Example \ref{counter_nesterov}.}
    \label{tab:vertical_table_nesterov}
    \begin{tabular}{ccc}  
        \toprule
        $k$ & $x^k$ & $\gamma_k$ \\  
        \midrule
        1 & 10 & 0.141421356237310 \\  
        2 & 8.58578643762690 & 0.116471566962991 \\  
        3 & 7.58578643762690 & 0.107635060338339 \\  
        4 & 6.76928985669918 & 0.104458044515078 \\  
        5 & 6.06218307551263 & 0.104328015857587 \\  
        $\cdots$ & $\cdots$ & $\cdots$ \\  
        13 & 2.06458695099841 & 0.189980988733214 \\  
        14 & 1.67235468072204 & 0.226007363967817 \\  
        $\cdots$ & $\cdots$ & $\cdots$ \\  
        24 & 0.209552285731976 & 1.37758046201432 \\  
        25 & -0.0791228488628367 & 3.57472862187939 \\  
        $\cdots$ & $\cdots$ & $\cdots$ \\  
        48 & 0.166305589462573 & 1.22740399701280 \\  
        49 & -0.0378185557693590 & 5.34210005645243 \\  
        $\cdots$ & $\cdots$ & $\cdots$ \\  
        60 & 0.155379438403268 & 1.17502153252226 \\  
        61 & -0.0271947474317873 & 6.65832593368331 \\  
        $\cdots$ & $\cdots$ & $\cdots$ \\  
        80 & 0.143015997988010 & 1.10556780523025 \\  
        81 & -0.0150978850204088 & 10.4077385707513 \\  
        \bottomrule
    \end{tabular}
    \vspace{-0.5em}  
\end{table}
\end{example}

Let $x^*$ be an optimal solution of \eqref{convex_problem}. Since $x^*$ is unknown, it is hard to precisely determine how $x^1$ should be chosen to satisfy the assumption \cite{alkousa2024mirror}
\begin{equation*}
    \max_{x \in Q} V_{\psi} (x^*, x) \leq V_{\psi} (x^*, x^1) < \infty.
\end{equation*}
Therefore, we restate the assumption as follows.   
\begin{assumption}\label{assump1}
There exists $R > 0$ such that 
\begin{equation}
    V_\psi(x^*,x)\le R<\infty \quad \forall x  \in Q.
\end{equation}
\end{assumption}

\subsection{Lipschitz-free mirror descent method}
Another important assumption in \cite[Theorem 3.1]{alkousa2024mirror}  is that the convex function $f$ should be $M_f$-Lipschitz, i.e., satisfy the following assumption.
\begin{assumption}\label{ass}
There exists an $M_f>0$ such that  for any subgradient $\nabla f(x)\in\partial f(x)\neq\emptyset$ and $x\in Q$, it holds that $\Vert \nabla f(x)\Vert_*\le M_f$.
\end{assumption}

Although one of the advantages of the adaptive step size rules is that there is no need to know the Lipschitz coefficient $M_f$ of the objective function in advance, Assumption \ref{assump1} is still necessary for the proof of the convergence result (see in \cite{nesterov2004lectures,nesterov2018lectures,bubeck2015convex}). However, many convex functions fail to satisfy Assumption \ref{ass} on a compact convex set.


In this section, we aim to offer a corrective proof for the optimal convergence result of the mirror descent method under the adaptive step size rules presented in \cite[Corollary 3.3]{alkousa2024mirror}. Moreover, we extend the optimal convergence rate of the mirror gradient descent method to the case where the Lipschitz condition is not assumed. Building upon the concept of weak ergodic convergence, we now introduce the following theorem.

\begin{theorem}\label{Lips_mirror_weak_ergodic}
Let $f$ be a convex continuous function, then for problem \eqref{convex_problem}, by Algorithm \ref{alg_mirror_descent}, for any fixed ${a}\in[0,1]$ and $m\ge-1$, with positive non-increasing step sizes
\begin{equation}\label{step_size}
    \gamma_k=\frac{\sqrt{2\sigma R}}{G_kk^{\frac{{a}}{2}}} ,\quad k=1,\ldots,N ,
\end{equation}
where
\begin{equation}\label{Gk}
    G_k=\max \left\{G_{k-1},\Vert\nabla f(x^k)\Vert_*k^{\frac{1-{a}}{2}}\right\} \quad(G_0 :=-\infty),
\end{equation}
and the weak ergodic convergence weight is defined as
\begin{equation}\label{weak_coeff}
    \omega_k^{(m)}=\begin{cases}
    1/\gamma_k^m, & \text{if } -1\le m\le 0, \\
     k^{m/2}, & \text{if } m>0,
    \end{cases}
\end{equation}
it satisfies the following inequality
\begin{equation}\label{result_weak_ergodic}
    f(\widehat{x}) -f(x^*)\le \sqrt{\frac{R}{2\sigma}} \cdot \frac{N^{\frac{m+1}{2}} + \sum_{k=1}^Nk^{\frac{m-1}{2}}}{\sum_{k=1}^N k^{\frac{m}{2}}} \cdot \max_{k=1, \ldots,N} \Vert\nabla f(x^k)\Vert_*,
\end{equation}
where 
{\color{black}
$$
    \widehat{x} = \frac{1}{\sum_{k=1}^N\omega_k^{(m)}}\sum_{k=1}^N\omega_k^{(m)}x^k.
$$}
\end{theorem}

\begin{proof}
Because $f$ is a convex function, we can obtain
\begin{eqnarray}
    \gamma_k(f(x^k)-f(x^*))&\le&\gamma_k\langle\nabla f(x^k),x^k-x^*\rangle\nonumber\\
    &=&\gamma_k\langle\nabla f(x^k),x^{k+1}-x^*\rangle+\gamma_k\langle\nabla f(x^k),x^k-x^{k+1}\rangle. \label{convex_f}
\end{eqnarray}

By Lemma \ref{prox_non}, and choosing $\varphi(x)=\gamma_k\langle \nabla f(x^k),x\rangle$, $a=x^k$, $b=x^{k+1}$, $u=x^*$, we can rewrite \eqref{eq_main_lemma} as
\begin{equation}\label{imply_prox}
    \left\langle \nabla \psi(x^k) - \nabla \psi(x^{k+1}), x^* - x^k \right\rangle \leq \gamma_k\langle \nabla f(x^k),x^*-x^{k+1}\rangle.
\end{equation}

From Lemma \ref{three_point}, we get
\begin{equation}\label{imply_three}
    \left\langle \nabla \psi(x^k) - \nabla \psi(x^{k+1}), x^* - x^k \right\rangle =V_{\psi}(x^*,x^{k+1})+V_{\psi}(x^{k+1}, x^k)-V_{\psi}(x^*, x^k). 
\end{equation}

Combining \eqref{imply_prox} and \eqref{imply_three}, and since $\psi$ is $\sigma$-strongly convex, we have
\begin{eqnarray}\label{three-sima}
    \gamma_k\langle \nabla f(x^k),x^{k+1}-x^*\rangle&\le& V_{\psi}(x^*, x^k)-V_{\psi}(x^*,x^{k+1})-V_{\psi}(x^{k+1}, x^k)\nonumber\\ 
    &\le&V_{\psi}(x^*, x^k)-V_{\psi}(x^*,x^{k+1})-\frac{\sigma}{2}\Vert x^{k+1}-x^k\|^2. \label{gamma-inq_1}
\end{eqnarray}
Based on the Fenchel-Young inequality, we  get
\begin{equation}\label{gamma_inq_2}
    \gamma_k\langle\nabla f(x^k),x^k-x^{k+1}\rangle\le\frac{\gamma_k^2}{2\sigma}\|\nabla f(x^k)\|^2_*+\frac{\sigma}{2}\|x^k-x^{k+1}\|^2.
\end{equation}
Combining \eqref{convex_f}, \eqref{gamma-inq_1} and \eqref{gamma_inq_2}, we obtain
\begin{equation}\label{eq_start}
    f(x^k) - f(x^*) \leq \frac{1}{\gamma_k} \left( V_{\psi} (x^*, x^{k}) - V_{\psi} (x^*, x^{k+1}) \right) + \frac{\gamma_k}{2 \sigma} \|\nabla f(x^k)\|_*^2.
\end{equation}

By the definition of weak ergodic convergence weight in \eqref{weak_coeff}, because $\{\gamma_k\}_{k=1}^N$ is a positive non-increasing step sizes, we have
\begin{equation}\label{+omega}
    \frac{\omega_k^{(m)}}{\gamma_k}-\frac{\omega_{k-1}^{(m)}}{\gamma_{k-1}}=\begin{cases}
    \frac{1}{\gamma_k^{m+1}}-\frac{1}{\gamma_{k-1}^{m+1}}\ge0, & \text{if } -1\le m\le0,k\ge2, \\
    \frac{k^{m/2}}{\gamma_k}-\frac{(k-1)^{m/2}}{\gamma_{k-1}}\ge0, & \text{if } m>0,k\ge2.
    \end{cases}
\end{equation}
We can rewrite $G_k$ in the following equivalent form
\begin{equation}\label{rewrite_GK}
    G_k=\max_{j=1,\ldots,k}\left\{\Vert\nabla f(x^j)\Vert_*\cdot j^{\frac{1-{a}}{2}}\right\}. 
\end{equation}
Then by Jensen's inequality, \eqref{eq_start}, the boundedness of $V_\psi(x^*,x^k)$, and \eqref{+omega}, we obtain
\begin{eqnarray}
    & &\left(\sum_{k=1}^N\omega_k^{(m)}\right)\left[f\left(\frac{\sum_{k=1}^N\omega_k^{(m)}x^k}{\sum_{k=1}^N\omega_k^{(m)}}\right)-f(x^*)\right] \nonumber\\
    &\le& \sum_{k=1}^N \omega_k^{(m)} \left(f(x^k)-f(x^*)\right)\nonumber\\ 
    &\le&\sum_{k=1}^N\frac{\omega_k^{(m)}}{\gamma_k} \left(V_\psi(x^*,x^k)-V_\psi(x^*,x^{k+1})\right)+\sum_{k=1}^N\frac{\omega_k^{(m)}\gamma_k}{2\sigma}\|\nabla f(x^k)\|_*^2 \nonumber\\
    &=&\frac{\omega_1^{(m)}}{\gamma_1}V_\psi(x^*,x^{1})+\sum_{k=2}^N \left(\frac{\omega_k^{(m)}}{\gamma_k}-\frac{\omega_k^{(m)}}{\gamma_{k-1}}\right) V_\psi(x^*,x^{k})\nonumber\\
    & &-\frac{\omega_N^{(m)}}{\gamma_N}V_\psi(x^*,x^{N})+\sum_{k=1}^N\frac{\omega_k^{(m)}\gamma_k}{2\sigma}\|\nabla f(x^k)\|_*^2 \nonumber\\
    &\le& R \left(\frac{\omega_1^{(m)}}{\gamma_1}+\sum_{k=2}^N \left(\frac{\omega_k^{(m)}}{\gamma_k}-\frac{\omega_k^{(m)}}{\gamma_{k-1}}\right)\right)+\sum_{k=1}^N\frac{\omega_k^{(m)}\gamma_k}{2\sigma}\|\nabla f(x^k)\|_*^2 \nonumber\\
    &\le&R\frac{\omega_N^{(m)}}{\gamma_N}+\sum_{k=1}^N\frac{\omega_k^{(m)}\gamma_k}{2\sigma}\|\nabla f(x^k)\|_*^2. \label{golden}
\end{eqnarray}
{\color{black}Let $\widehat{x} := \frac{\sum_{k=1}^N\omega_k^{(m)}x^k}{\sum_{k=1}^N\omega_k^{(m)}}$.} Because the weak ergodic convergence weight is non-negative, for $-1\le m\le0$, by the definition of $\gamma_k$ in \eqref{step_size} and \eqref{Gk}, we have, 
\begin{eqnarray}
    &&f(\widehat{x})-f(x^*)\nonumber\\
    &\le&\frac{R\frac{\omega_N^{(m)}}{\gamma_N}+\sum_{k=1}^N\frac{\omega_k^{(m)}\gamma_k}{2\sigma}\|\nabla f(x^k)\|_*^2}{\sum_{k=1}^N\omega_k^{(m)}} \nonumber\\
    &\le&\frac{\frac{R}{\gamma_N^{m+1}}+\sum_{k=1}^N\frac{\gamma_k^{1-m}}{2\sigma}\|\nabla f(x^k)\|_*^2}{\sum_{k=1}^N\gamma_k^{-m}} \nonumber\\
    &\le&\sqrt{\frac{R}{2\sigma}}\frac{\left(G_NN^{\frac{{a}}{2}}\right)^{m+1}+\sum_{k=1}^N \left(G_kk^{\frac{{a}}{2}}\right)^{m-1}\|\nabla f(x^k)\|_*^2}{\sum_{k=1}^N \left(G_kk^{\frac{{a}}{2}}\right)^m} \nonumber\\
    &\le& \sqrt{\frac{R}{2\sigma}} \frac{\left(\max\limits_{k=1,\ldots,N}\Vert\nabla f(x^k)\Vert_* \right)^{1+m}\cdot N^{\frac{m+1}{2}}+\sum_{k=1}^N\|\nabla f(x^k)\|_*^{1+m} \cdot k^{\frac{m-1}{2}}}{\left(\max\limits_{k=1,\ldots,N}\Vert\nabla f(x^k)\Vert_* \right)^m\sum_{k=1}^Nk^{\frac{m}{2}}} \nonumber\\
    &\le&\sqrt{\frac{R}{2\sigma}}\cdot\frac{N^{\frac{m+1}{2}}+\sum_{k=1}^Nk^{\frac{m-1}{2}}}{\sum_{k=1}^Nk^{\frac{m}{2}}}\cdot\max_{k=1,\ldots,N}\Vert\nabla f(x^k)\Vert_*\label{minor_m}.
\end{eqnarray}
When $m>0$, by the definition of $\omega_k^{(m)}$ and $\gamma_k$,  we can get
\begin{eqnarray*}
    && f(\widehat{x})-f(x^*)\\
    &\le&\frac{\frac{R\cdot N^{\frac{m}{2}}}{\gamma_N}+\sum_{k=1}^N\frac{k^{\frac{m}{2}}\gamma_k}{2\sigma}\|\nabla f(x^k)\|_*^2}{\sum_{k=1}^Nk^{\frac{m}{2}}}\\
    &\le&\sqrt{\frac{R}{2\sigma}}\frac{G_NN^{\frac{m+{a}}{2}}+\sum_{k=1}^N\|\nabla f(x^k)\|_*^2G_k^{-1}k^{\frac{m-{a}}{2}}}{\sum_{k=1}^Nk^{\frac{m}{2}}} \\
    &\le&\sqrt{\frac{R}{2\sigma}}\frac{\max\limits_{k=1,\ldots,N}\Vert\nabla f(x^k)\Vert_*\cdot N^{\frac{m+1}{2}}+\sum_{k=1}^N\|\nabla f(x^k)\|_*\cdot k^{\frac{m-1}{2}}}{\sum_{k=1}^Nk^{\frac{m}{2}}} \\
    &\le&\sqrt{\frac{R}{2\sigma}}\cdot\frac{N^{\frac{m+1}{2}}+\sum_{k=1}^Nk^{\frac{m-1}{2}}}{\sum_{k=1}^Nk^{\frac{m}{2}}}\cdot\max_{k=1,\ldots,N}\Vert\nabla f(x^k)\Vert_* .
\end{eqnarray*}
This completes the proof.
\qed
\end{proof}

By setting $m=0$, we can obtain the optimal convergence rate result of the Lipschitz-free mirror descent method under the adaptive time-varying step size rules:
\begin{corollary}\label{linear_ergodic_result}
Let $f$ be a convex continuous function, then for problem \eqref{convex_problem}, by Algorithm \ref{alg_mirror_descent}, for any fixed ${a}\in[0,1]$ and $m=0$, with the step sizes \eqref{step_size}, it satisfies the following inequality
\begin{equation}\label{linear_eq}
    f\left(\frac{1}{N}\sum_{k=1}^Nx^k \right)-f(x^*)\le3\sqrt{\frac{R}{2\sigma}}\cdot\frac{1}{\sqrt{N}}\cdot\max_{k=1,\ldots,N}\Vert\nabla f(x^k)\Vert_*.
\end{equation}
\end{corollary}
\begin{remark}
If we choose the distance generating function $\psi=\frac{1}{2}\|x\|_2^2$, then the results in Theorem \ref{Lips_mirror_weak_ergodic} and Corollary \ref{linear_ergodic_result} reduce to the result in \cite{Lipschitz-free}. If further given Assumption \ref{ass}, Corollary \ref{linear_ergodic_result} would be the same as \cite[Corollary 3.2]{Zhu}.
\end{remark}

{\color{black}
\section{Lipschitz-free mirror descent method for constrained composite
problem}
}

In this section, we consider the composite convex problem
\begin{equation}\label{composite}
    \min_{x\in Q} F(x)=f(x)+h(x),
\end{equation}
where $f$ is convex function, and \(h\) is a non-negative, continuous, and convex function. Compared with Algorithm \ref{alg_mirror_descent}, we just need to change the iterative strategy to the following 

\begin{algorithm}[!ht]
\caption{Composite Mirror Descent Method.}\label{alg_mirror_descent_com}
\begin{algorithmic}[2]
\REQUIRE Step sizes $\{\gamma_k\}_{k \geq 1}$,  initial point $x^1  \in Q$, number of iterations $N$.
\FOR{$k= 1, 2, \ldots, N$}
\STATE Calculate $\nabla f(x^k) \in \partial f(x^k)$,
\STATE $    x^{k+1} = \arg\min_{x \in Q} \left\{ \langle \nabla f(x^k),x  \rangle  +\frac{1}{\gamma_k}h(x)+ \frac{1}{\gamma_k} V_{\psi}(x,x^k)\right\}$.
\ENDFOR
\end{algorithmic}
\end{algorithm}

Now we introduce the following theorem, while Assumption \ref{ass} is still not necessary.
\begin{theorem}\label{comsite_ergodic}
Let $f$ be a convex continuous function and $h$ be a non-negative convex function, then for problem \eqref{composite}, by Algorithm \ref{alg_mirror_descent_com}, for any fixed ${a}\in[0,1]$ and $-1\le m\le0$, with positive non-increasing step sizes $\{\gamma_k\}_{k=1}^N$ defined in \eqref{step_size} and weak ergodic convergence weight $\left\{\omega_k^{(m)}\right\}_{k=1}^N$ defined in \eqref{weak_coeff}, it satisfies the following inequality
\begin{align*}
    F(\widehat{x})-F(x^*) & \le\sqrt{\frac{R}{2\sigma}} \cdot\frac{N^{\frac{m+1}{2}} +\sum_{k=1}^Nk^{\frac{m-1}{2}}}{\sum_{k=1}^Nk^{\frac{m}{2}}}\cdot\max_{k=1,\ldots,N}\Vert\nabla f(x^k)\Vert_*
    \\& \quad 
    +\left(\frac{\|\nabla f(x^1)\|_*}{\max\limits_{k=1,\ldots,N} \Vert\nabla f(x^k)\Vert_*}\right)^m \cdot \frac{h(x^1)}{\sum_{k=1}^Nk^{\frac{m}{2}}},
\end{align*}
where $\widehat{x} = \frac{1}{\sum_{k=1}^N\omega_k^{(m)}} \sum_{k=1}^N\omega_k^{(m)}x^k$.
\end{theorem}

\begin{proof}
By Lemma \ref{prox_non}, by choosing $\varphi(x)=\gamma_k\langle \nabla f(x^k),x\rangle+\gamma_kh(x)$, $a=x^k$, $b=x^{k+1}$, $u=x^*$, then \eqref{eq_main_lemma} can be rewritten as
\begin{equation}\label{imply_prox_com}
    \left\langle \nabla \psi(x^k) - \nabla \psi(x^{k+1}), x^* - x^k \right\rangle \leq \gamma_k\langle \nabla f(x^k),x^*-x^{k+1}\rangle+\gamma_kh(x^*)-\gamma_kh(x^{k+1}).
\end{equation}
Combining \eqref{imply_prox_com}, \eqref{gamma-inq_1} and Lemma \ref{three_point}, we can get
\begin{align}\label{h+three}
    \gamma_k\langle \nabla f(x^k),x^{k+1}-x^*\rangle & \le V_{\psi}(x^*, x^k)-V_{\psi}(x^*,x^{k+1})-\frac{\sigma}{2}\Vert x^{k+1}-x^k\|^2\nonumber
    \\& \quad +\gamma_kh(x^*)-\gamma_kh(x^{k+1}).
\end{align}
Note that \eqref{convex_f} and \eqref{gamma_inq_2} still hold, so we have
\begin{eqnarray}
    \gamma_k(f(x^k)+h(x^{k+1})-F(x^*))=\gamma_k(f(x^k)-f(x^*)+h(x^{k+1})-h(x^*))\nonumber\\
    \le V_{\psi}(x^*, x^k)-V_{\psi}(x^*,x^{k+1})+\frac{\gamma_k^2}{2 \sigma} \|\nabla f(x^k)\|_*^2.
\end{eqnarray}
Then we can obtain the following important inequality
\begin{equation}
    f(x^k)+h(x^{k+1})-F(x^*)\le\frac{1}{\gamma_k} \left( V_{\psi} (x^*, x^{k}) - V_{\psi} (x^*, x^{k+1}) \right) + \frac{\gamma_k}{2 \sigma} \|\nabla f(x^k)\|_*^2.
\end{equation}

Similar to the analysis in \eqref{golden}, since the $\gamma_k$ we selected is exactly the same, we get
\begin{equation}\label{ommga_com}
    \sum_{k=1}^N \omega_k^{(m)} \left(f(x^k)+h(x^{k+1})-F(x^*)\right) \le R\frac{\omega_N^{(m)}}{\gamma_N} + \sum_{k=1}^N \frac{\omega_k^{(m)} \gamma_k}{2\sigma} \|\nabla f(x^k)\|_*^2.
\end{equation}

Now, we shall handle the left-hand side of \eqref{ommga_com}. Because $\omega_k^{(m)}$ is non-increasing when $-1\le m\le 0$ and $h$ is non-negative, we have
\begin{eqnarray}
    && \quad \sum_{k=1}^N \omega_k^{(m)} \left(f(x^k)+h(x^{k+1})-F(x^*)\right) \nonumber\\
    && = \sum_{k=1}^N \omega_k^{(m)} \left(f(x^k)+h(x^k)-h(x^k)+h(x^{k+1})-F(x^*)\right) \nonumber\\
    && = \sum_{k=1}^N \omega_k^{(m)} \left( F(x^k) -F(x^*) \right)+ \sum_{k=1}^N \omega_k^{(m)}\left(h(x^{k+1}) - h(x^k) \right) \nonumber\\
    && = \sum_{k=1}^N \omega_k^{(m)} \left(F(x^k) - F(x^*) \right)-\omega_1^{(m)}h(x^1) + \sum_{k=2}^N \left(\omega_{k-1}^{(m)} -\omega_k^{(m)}\right)h(x^k) + \omega_N^{(m)} h(x^{N+1}) \nonumber\\
    && \geq \sum_{k=1}^N \omega_k^{(m)} \left(F(x^k)-F(x^*) \right)-\omega_1^{(m)}h(x^1).
\end{eqnarray}
Then by Jensen's Inequality, we have:
\begin{eqnarray*}
    & &   \left(\sum_{k=1}^N\omega_k^{(m)}\right)\left[F\left(\frac{\sum_{k=1}^N\omega_k^{(m)}x^k}{\sum_{k=1}^N\omega_k^{(m)}}\right)-F(x^*)\right] \\ 
    &\le&\sum_{k=1}^N\omega_k^{(m)}(F(x^k)-F(x^*)) \\ 
    &\le&\sum_{k=1}^N\omega_k^{(m)}(f(x^k)+h(x^{k+1})-F(x^*))+\omega_1^{(m)}h(x^1)
    \\
    &\le&R\frac{\omega_N^{(m)}}{\gamma_N}+\sum_{k=1}^N\frac{\omega_k^{(m)}\gamma_k}{2\sigma}\|\nabla f(x^k)\|_*^2+\omega_1^{(m)}h(x^1).
\end{eqnarray*}
According to the analysis in \eqref{minor_m}, we know that the result below has been established when $-1\le m\le0$
\begin{align*}
    & \quad \frac{R\frac{\omega_N^{(m)}}{\gamma_N}+\sum_{k=1}^N\frac{\omega_k^{(m)}\gamma_k}{2\sigma}\|\nabla f(x^k)\|_*^2}{\sum_{k=1}^N\omega_k^{(m)}}
    \\& \le\sqrt{\frac{R}{2\sigma}}\cdot\frac{N^{\frac{m+1}{2}}+\sum_{k=1}^Nk^{\frac{m-1}{2}}}{\sum_{k=1}^Nk^{\frac{m}{2}}}\cdot\max_{k=1,\ldots,N}\Vert\nabla f(x^k)\Vert_*. 
\end{align*}
Recall that $G_k$ can be rewritten as \eqref{rewrite_GK}, so we get
\begin{align*}
    \frac{\omega_1^{(m)}h(x^1)}{\sum_{k=1}^N\omega_k^{(m)}}&=\frac{\frac{1}{\gamma_k^m}h(x^k)}{\sum_{k=1}^N\frac{1}{\gamma_k^m}}=\frac{G_1^mh(x^1)}{\sum_{k=1}^N(G_kk^\frac{{a}}{2})^m}
    \\&=\frac{G_1^mh(x^1)}{\sum_{k=1}^N\left(\max\limits_{j=1,\ldots,k}(\Vert\nabla f(x^j)\Vert_*\cdot j^{\frac{1-{a}}{2}})k^\frac{{a}}{2}\right)^m}
    \\& \le\frac{\|\nabla f(x^1)\|_*^mh(x^1)}{\sum_{k=1}^N\left(\max\limits_{k=1,\ldots,N}\Vert\nabla f(x^k)\Vert_*\cdot k^{\frac{1-{a}}{2}}\cdot k^\frac{{a}}{2}\right)^m}
    \\& =\left(\frac{\|\nabla f(x^1)\|_*}{\max\limits_{k=1, \ldots,N}\Vert\nabla f(x^k)\Vert_*}\right)^m\cdot\frac{h(x^1)}{\sum_{k=1}^Nk^{\frac{m}{2}}}.
\end{align*}
Thus, by setting $\widehat{x}  = \frac{1}{\sum_{k=1}^N\omega_k^{(m)}} \sum_{k=1}^N\omega_k^{(m)}x^k$,  we derive the following inequality
\begin{align*}
    F (\widehat{x}) -F(x^*) & \le\sqrt{\frac{R}{2\sigma}} \cdot\frac{N^{\frac{m+1}{2}} + \sum_{k=1}^N k^{\frac{m-1}{2}}}{\sum_{k=1}^Nk^{\frac{m}{2}}}\cdot\max_{k=1,\ldots,N}\Vert\nabla f(x^k)\Vert_*
    \\& \quad 
    +\left(\frac{\|\nabla f(x^1)\|_*}{\max\limits_{k=1,\ldots,N}\Vert\nabla f(x^k)\Vert_*}\right)^m \cdot\frac{h(x^1)}{\sum_{k=1}^N k^{\frac{m}{2}}}.
\end{align*}
This completes the proof.\qed
\end{proof}

By setting $m = 0$, we can obtain the following corollary, which reveals the optimal convergence rate of the composite mirror descent method under the adaptive time-varying step size rules.
\begin{corollary}\label{linear_ergodic_result_com}
Let $f$ be a convex continuous function and $h$ be a non-negative convex function, then for problem \eqref{composite}, by Algorithm \ref{alg_mirror_descent_com}, for any fixed ${a}\in[0,1]$ and $m=0$, with the step sizes \eqref{step_size}, it satisfies the following inequality
\begin{equation}\label{linear_eq}
    F\left(\frac{1}{N} \sum_{k=1}^Nx^k \right)-F(x^*)\le3\sqrt{\frac{R}{2\sigma}}\cdot\frac{1}{\sqrt{N}}\cdot\max_{k=1,\ldots,N}\Vert\nabla f(x^k)\Vert_*+\frac{h(x^1)}{N}.
\end{equation}
\end{corollary}

\begin{remark}
In \cite[Corollary 4.3]{alkousa2024mirror}, it is required to select the initial point \(x^1\) satisfying \(h(x^1) = 0\), which is not necessary for our theory.
\end{remark}
\begin{remark}
     As revealed by the formula of the corollary, if we can separate a non-negative convex function in a convex optimization problem, then the part corresponding to the non-negative convex function $h(x)$ will converge at a faster rate $O(1/N)$, and this perspective is not mentioned in \cite{alkousa2024mirror}.
\end{remark}

\section*{Conclusion}

In this work, we correct the proofs in the part of the recent study \cite{alkousa2024mirror} connected with the analysis of the convergence rate of the mirror descent method and composite mirror descent method under adaptive time-varying step size rules. Furthermore, we introduce a Lipschitz-free mirror descent algorithm that achieves weak ergodic convergence, extending the results of mirror descent beyond the traditional Lipschitz continuity setting. This advancement broadens the understanding of mirror descent methods and opens new possibilities for their application in non-smooth and non-Lipschitz optimization problems.

\section*{Funding}
This research was supported by the National Natural Science Foundation of China under grant 12171021, and the Fundamental Research Funds for the Central Universities.

\section*{Data Availability}
The manuscript has no associated data.


\begin{thebibliography}{9}

\bibitem{alkousa2024mirror}
Alkousa, M., Stonyakin, F., Abdo, A., Alcheikh, M. (2024). Optimal Convergence Rate for Mirror Descent Methods with Special Time-Varying Step Sizes Rules. In: Eremeev, A., Khachay, M., Kochetov, Y., Mazalov, V., Pardalos, P. (eds) Mathematical Optimization Theory and Operations Research: Recent Trends. MOTOR 2024. Communications in Computer and Information Science, vol 2239. Springer, Cham. https://doi.org/10.1007/978-3-031-73365-9\_1.


\bibitem{Nemirovsky1983Complexity}
Nemirovsky, A., Yudin D.: Problem Complexity and Method Efficiency in Optimization. J. Wiley \& Sons, New York 1983.	

\bibitem{Nemirovskii1979efficient}
Nemirovskii, A.: Efficient methods for large-scale convex optimization problems. Ekonomika i Matematicheskie Metody, 1979. (in Russian)

\bibitem{Beck2003Mirror}
Beck, A., Teboulle, M.: Mirror descent and nonlinear projected subgradient methods for convex optimization. Oper. Res. Lett., 31(3), pp. 167--175, 2003.

\bibitem{article:doan_2019}
Doan, T.T., Bose, S., Nguyen, D.H., Beck, C.L.: Convergence of the Iterates in Mirror Descent Methods. IEEE Control Systems Letters,  \textbf{3}(1), pp. 114--119, 2019.


\bibitem{article:beck_comirror_2010}
Beck, A., Ben-Tal, A., Guttmann-Beck, N., Tetruashvili, L.: The comirror algorithm for solving nonsmooth constrained convex problems. Operations Research Letters, \textbf{38}(6), pp. 493--498, 2010.


\bibitem{applications_tomography_2001}
Ben-Tal, A., Margalit, T., Nemirovski, A.: The Ordered Subsets Mirror Descent Optimization Method with Applications to Tomography. SIAM Journal on Optimization, \textbf{12}(1), pp. 79--108, 2001.

\bibitem{article:Nazin_2014}
Nazin, A., Anulova, S., Tremba, A.: Application of the Mirror Descent Method to Minimize Average Losses Coming by a Poisson Flow. European Control Conference (ECC) June, pp. 24--27, 2014.

\bibitem{nesterov2004lectures}
Y. Nesterov, Introductory lectures on convex optimization: A basic course, Springer Science \& Business Media, 2004.

\bibitem{nesterov2018lectures}
Y. Nesterov, Lectures on convex optimization, volume 137, Springer, 2018.

\bibitem{bubeck2015convex}
S. Bubeck, Convex optimization: Algorithms and complexity. Foundations
and Trends Trends{\textregistered} in Machine Learning, 8(3-4): 231-357, 2015.

\bibitem{Zhu}
Z. Zhu, Y. Zhang, Y. Xia, Convergence rate of projected subgradient method
with time-varying step-sizes, Optim Lett 19, 1027–1031 (2025). https://doi.org/10.1007/s11590-024-02142-9.

\bibitem{Lipschitz-free}
Y. Xia, Y. Zhang, Z. Zhu, Lipschitz-free Projected Subgradient Method with Time-varying Step-size, 2024. https://optimization-online.org/?p=27932. arXiv:2410.22336.

\bibitem{beck2017first}
A.Beck, First-order methods in optimization. Society for Industrial and Applied Mathematics, 2017.
\bibitem{Lan}
G. Lan, First-order and Stochastic Optimization Methods for Machine Learning, Springer-Nature, 2020
\end{thebibliography}
\end{document}